\documentclass[12]{article}
\usepackage[ansinew]{inputenc}
\usepackage{amsmath,amsthm,amsfonts,enumerate}

\newtheorem{theorem}{Theorem}[section]
\newtheorem{lemma}[theorem]{Lemma}

\newtheorem{corollary}[theorem]{Corollary}

\newtheorem{result}[theorem]{Result}

\newtheorem{remarks}[theorem]{Remarks}

\newcommand{\erz}[1]{\langle #1\rangle}

\DeclareMathOperator{\tw}{tw}
\DeclareMathOperator{\PG}{PG}
\newcommand{\gauss}[2]{\ensuremath{\genfrac{[}{]}{0pt}{1}{#1}{#2}}}
\renewcommand{\gauss}[2]{{#1\brack #2}_q}

\date{}
\title{On the treewidth of generalized $q$-Kneser graphs}
\author{Klaus Metsch\thanks{Justus-Liebig-Universit\"{a}t, Mathematisches Institut,
Arndtstra{\ss}e 2, D-35392 Gie{\ss}en}}

\begin{document}
\maketitle

\begin{abstract}
The generalized $q$-Kneser graph $K_q(n,k,t)$ for integers $k>t>0$ and $n>2k-t$ is the graph whose vertices are the $k$-dimensional subspaces of an $n$-dimensional $F_q$-vectorspace with two vertices $U_1$ and $U_2$ adjacent if and only if $\dim(U_1\cap U_2)<t$. We determine the treewidth of the generalized $q$-Kneser graphs $K_q(n,k,t)$ when $t\ge 2$ and $n$ is sufficiently large compared to $k$. The imposed bound on $n$ is a significant improvement of the previously known bound. One consequence of our results is that the treewidth of each $q$-Kneser graph $K_q(n,k,t)$ with $k>t>0$ and $n\ge 3k-t+9$ is equal to $\gauss{n}{k}-\gauss{n-t}{k-t}-1$.
\end{abstract}

\textbf{Keywords:} generalized Kneser graph, treewidth, tree decomposition

\textbf{MSC (2020):} 51D05, 51E20

\section{Introduction}

In this paper we only consider simple \emph{graphs} $\Gamma$ without loops, that is  $\Gamma$ is a pair $(X,E)$ where $X$ is a non-empty set and $E$ is a set of subsets of cardinality two of $X$. By graph we always refer to a simple graph without loops. The elements of $X$ are called \emph{vertices} and the elements of $E$ are called \emph{edges} of $\Gamma$. We write $X=V(\Gamma)$. A graph is \emph{empty} if it has no edges, and \emph{finite} if it has only finitely many vertices.

A \emph{tree decomposition} of a graph $\Gamma$ is a pair $(T,B)$ where $T$ is a tree and $B=(B_t:t\in V(T))$ is a collection of subsets $B_t$ of $V(\Gamma)$, indexed by the vertices of $T$, such that
\begin{enumerate}[(TD1)]
\item every edge $\{u,v\}$ of $\Gamma$ is contained in $B_t$ for some $t\in V(T)$, and
\item for each $v\in V(\Gamma)$, the set $\{t\in V(T)\mid v\in B_t\}$ is not empty and the graph induced by $T$ on this set is connected.
\end{enumerate}
Each graph $\Gamma$ has the trivial tree decomposition where the tree has just one vertex $t$ and $B_t=V(\Gamma)$. The \emph{width} of a tree decomposition of a finite graph is the number $\max\{|B_t|-1\mid t\in V(T)\}$, and the \emph{treewidth} $\tw(\Gamma)$ of a finite graph $\Gamma$ is the smallest width of its tree decompositions. The treewidth of a graph measures how treelike a graph is. For example, the treewidth of a non-empty tree is one and the treewidth of a graph on $n$ vertices is at most $n-1$ with equality if and only if the graph is complete. There is a vast literature on the treewidth of graphs, see \cite{c3,c5,c6,c2,Liu&Cao&Lu,c4,c1} for some recent ones,  and there are applications. A famous one is by Robertson and Seymour \cite{Robertson&Seymour} when they proved that the treewidth of a minor of a graph cannot exceed the treewidth of the graph.

Let $q$ be a prime power and let $F_q$ be the finite field of order $q$. For integers $n,k,t$ with $k>t\ge 1$ and $n>2k-t$, the \emph{$q$-Kneser graph} $K_q(n,k,t)$ is the graph whose vertices are the subspaces of dimension $k$ of the vector space $F_q^n$ with two vertices $K$ and $K'$ adjacent if and only if $\dim(K\cap K')<t$. The condition $n>2k-t$ ensures that the graph is non-empty.
Considering the vertices of $K_q(n,k,t)$ as subspaces of the vector space dual to $F_q^n$, one shows that the graphs $K_q(n,k,t)$ and $K_q(n,n-k,n-2k+t)$ are isomorphic. It is therefore sufficient to consider Kneser graphs $K_q(n,k,t)$ with $n\ge 2k$. Notice that we do not allow $t=k$, which would yields only complete graphs.
In 2021 Cao, Liu, Lu and Lv proved the following theorem.

\begin{theorem}[\cite{CaoLiuLuLv}]
Let $q$ be the order of a finite field. For integers $n,k,t$ with $k>t\ge 1$ and
\begin{align}\label{Liubound}
n\ge 2t(k-t+1)+k+1
\end{align}
we have $\tw(K_q(n,k,t))=\gauss{n}{k}-\gauss{n-t}{k-t}-1$.
\end{theorem}

In this result the Gaussian coefficient $\gauss{n}{k}$ is defined for integers $n\ge k\ge 0$ and $q\ge 2$ as follows.
\begin{align}
\gauss{n}{k}:=\prod_{i=0}^{k-1}\frac{q^{n-i}-1}{q^i-1}.
\end{align}

The hard part of the theorem is to prove the lower bound for the treewidth. For the upper bound, the authors of \cite{CaoLiuLuLv} proved in fact that $\tw(K_q(n,k,t))\le \gauss{n}{k}-\gauss{n-t}{k-t}-1$ holds whenever $k>t\ge 1$ and $n\ge 2k$. In this paper we give a shorter proof for this upper bound and we will weaken the required condition \eqref{Liubound} for the lower bound substantially. Our main result is the following. To formulate it, we define the map $\epsilon$ from the set of all primepowers to the set $\{0,1,2,3,4\}$ by $\epsilon(2)=9$, $\epsilon(3)=3$, $\epsilon(4)=2$, $\epsilon(q)=1$ for $5\le q\le 8$, and $\epsilon(q)=0$ for $q\ge 9$.

\begin{theorem}\label{main1}
Consider a $q$-Kneser graph $K_q(n,k,t)$ with $k>t\ge 1$ and $n\ge 2k$. Suppose that one of the following condition is satisfied.
\begin{enumerate}[(1)]
\item $t\le \epsilon(q)$ and $n>3k-2t+\epsilon(q)$.
\item $t>\epsilon(q)$ and $n>3k-t+1-2\sqrt{t-\epsilon(q)}$.
\end{enumerate}
Then
$\tw(K_q(n,k,t))=\gauss{n}{k}-\gauss{n-t}{k-t}-1$.
\end{theorem}

The strongest result is obtained for $q\ge 9$, where the required bound is only $n>3k-t+1-2\sqrt{t}$. As a simpler formulation not involving $\epsilon(q)$ we can state the following.

\begin{corollary}\label{maincor1}
For all integers $k>t>0$ and $n\ge 3k-t+9$ and any prime power $q$, the  $q$-Kneser graph $K_q(n,k,t)$ with has treewidth $\gauss{n}{k}-\gauss{n-t}{k-t}-1$.
\end{corollary}

For some pairs $(k,t)$ and fields $F_q$ our result determines the treewidth for all $n\ge 2k$. We formulate this in the following corollary. For simplicity we only consider the case when $q$ is at least nine.

\begin{corollary}\label{maincor2}
Let $q\ge 9$ be the order of a finite field and let $k$ and $t$ be integers with $k>t\ge 1$, and $t>k+3-2\sqrt{k+2}$. Then
$$\tw(K_q(n,k,t))=\gauss{n}{k}-\gauss{n-t}{k-t}-1$$ for all $n\ge 2k$.
\end{corollary}

\begin{remarks}
\begin{enumerate}
\item For the non-modular Kneser graphs $K(n,k,t)$ similar results have been proved in \cite{Harvey&Wood}, \cite{Liu&Cao&Lu} and     \cite{Metsch}.
\item As was shown in \cite{CaoLiuLuLv}, see also Corollary \ref{upperbound}, we have $\tw(\Gamma)\le |V(\Gamma)|-\alpha(\Gamma)-1$ for every $q$-Kneser graph $\Gamma$. Notice however that the independence number of $K_q(n,k,t)$ is $\gauss{n-t}{k-t}$ when $n\ge 2k$ and $\gauss{2k-t}{k-t}$ otherwise.
\end{enumerate}
\end{remarks}

The Kneser graph $K_q(n,k,k-1)$ is the dual of the Grassmann graph. Part (2) of Theorem \ref{main1} determines its treewidth for all $k\ge 3+\epsilon(q)$ and $n\ge 2k$. This was proved in \cite{CaoLiuLuLv} more generally for all $k\ge 3$ and also in the case $k=2$ and $n\ge 5$. The case $K_q(4,2,1)$ remained open in \cite{CaoLiuLuLv}. Our last results settles this remaining open case. This requires a more detailed look at tree decompositions.

\begin{theorem}
The treewidth of $K_q(4,2,1)$ is $|V(\Gamma)|-(q^2+q+2)$ for all prime powers $q\ge 2$.
\end{theorem}

\section{An upper bound for the treewidth}\label{sectionupperbound}

For every finite graph $\Gamma$ its maximum vertex degree is denoted by $\Delta(\Gamma)$ and is called its \emph{maximum degree}. The cardinality of a largest independent set of $\Gamma$ is denoted by $\alpha(\Gamma)$ and is called the \emph{independence number} of the graph. The following connection between treewidth, maximum degree and independence number is known.

\begin{result}[\cite{Harvey&Wood}]\label{Harvey&Wood}
For every finite graph $\Gamma$ we have
\begin{align}\label{twboundwithDelta}
\tw(\Gamma)\le\max\{\Delta(\Gamma),|V(\Gamma)|-\alpha(\Gamma)-1\}.
\end{align}
\end{result}

For generalized $q$-Kneser graphs we have $\Delta(\Gamma)\le |V(\Gamma)|-\alpha(\Gamma)-1$ as was shown for $n\ge 2k$ in \cite{CaoLiuLuLv}. We give a simpler proof for this inequality. In order to do so, we need the following result of Wilson on the independence number of generalized Kneser graphs.

\begin{result}[\cite{Frankl&Wilson}]\label{Frankl&Wilson}
For integers $k>t\ge 1$ and $n>2k-t$ we have
\begin{align*}
\alpha(K_q(n,k,t))=max\{\gauss{n-t}{k-t},\gauss{2k-t}{k-t}\}.
\end{align*}
\end{result}

\begin{lemma}
For Kneser graphs $K_q(n,k,t)$ with $n>2k-t$ and $k>t\ge 1$ we have
\begin{align*}
\Delta(K_q(n,k,t))\le\gauss{n}{k}-\alpha(K_q(n,k,t))-1.
\end{align*}
\end{lemma}
\begin{proof}
Since $K_q(n,k,t)$ and $K_q(n,n-k,n-2k+t)$ are isomorphic graphs, we may assume that $n\ge 2k$. The graph $\Gamma:=K_q(n,k,t)$ has $\gauss{n}{k}$ vertices and is regular. Given a vertex $A$, then $r:=\gauss{n}{k}-\Delta(\Gamma)$ is the number of vertices of the graph that are not adjacent to it. Considering $A$ as a $k$-subspace of $F_q^n$, then $r$ is the number of $k$-subspaces meeting $A$ in a subspace of dimension at least $t$. We have to show that this number is strictly larger than $\alpha(\Gamma)$.

From \ref{Frankl&Wilson} we have $\alpha(\Gamma)=\gauss{n-t}{k-t}$. Let $T_0$ be a $t$-subspace of $A$. Then $T$ is contained in $\gauss{n-t}{k-t}$$k$-subspaces. Let $T$ be a second $t$-subspace of $A$. Notice that $T$ exists as $t<k$. Then there exist $k$-subspaces meeting $A$ exactly in $T$. This shows that $r>\gauss{n-t}{k-t}=\alpha(\Gamma)$ as desired.
\end{proof}

\begin{corollary} \label{upperbound}
For Kneser graphs $K_q(n,k,t)$ with $n>2k-t$ and $1\le t<k$ we have
\begin{align*}
\tw(K_q(n,k,t))\le
\gauss{n}{k}-\alpha(K_q(n,k,t))-1.
\end{align*}
\end{corollary}

As mentioned before, it was proved in \cite{CaoLiuLuLv} that this upper bound is sharp when $n$ is sufficiently large compared to $k$ and $t$. We will improve this result by weakening the required bound on $n$ significantly. The approach in \cite{CaoLiuLuLv} and \cite{Harvey&Wood} uses a result of Robertson and Seymour on separators. We use the slightly different result (2.5) from the same paper and modify it in the following lemma. Recall that a component of a graph is a maximal connected induced subgraph of $\Gamma$. Hence, every vertex of the graph is a vertex of exactly one component.

\begin{lemma}[based on \cite{Robertson&Seymour}]\label{Robertson&Seymour}
Let $\Gamma$ be a finite graph. Then there exists a tree representation $(T,B)$ with $B=(B_t\mid t\in V(T))$ with width $\tw(\Gamma)$ and the following properties.
\begin{enumerate}[(a)]
\item If $t,t'\in T$ are adjacent in $T$, then $B_t$ is not a subset of $B_{t'}$.
\item For some $t\in T$, there exists a subset $P$ of $B_t$ such that $|P|\le \tw(\Gamma)+1$ and such that every component of $\Gamma\setminus P$ has at most $\frac12|V(\Gamma)\setminus P|$ vertices.
\end{enumerate}
\end{lemma}
\begin{proof}
Let $(T,B)$ with $B=(B_t\mid t\in V(T))$ be a tree decomposition of $\Gamma$  with width $\tw(\Gamma)$. If there exist adjacent vertices $x,y$ of $T$ with $B_x\subseteq B_y$, let $T'$ be the graph obtained from $T$ by contracting the edge $\{x,y\}$ to a new vertex $t_0$, and define $B_{t_0}:=B_x$. Then $(T',(B_t\mid t\in V(T'))$ is a tree decomposition of $\Gamma$ with the same width as $(T,B)$ and one vertex less than $T$. Repeating this construction several times if necessary we finally obtain a tree decomposition satisfying the first condition of the lemma and having still tree width $\tw(\Gamma)$. We may thus assume that $(T,B)$ satisfies the first condition.

Claim (2.5) in \cite{Robertson&Seymour} states that there exists a subset $P$ of $V(\Gamma)$ such that every component of $\Gamma\setminus P$ has at most $\frac12|V(\Gamma)\setminus P|$ vertices. Using the above tree decomposition, the proof of (2.5) in \cite{Robertson&Seymour} in fact constructs such a set $P$ as a subset of some set $B_t$. Thus the second condition is satisfied.
\end{proof}

\section{A lower bound for the treewidth}

\begin{lemma}\label{guassestimate}
\label{boundsgauss_simple}
Consider integers $n,k,q$ with $n\ge k\ge 0$ and $q\ge 2$. Then
\begin{align*}
\gauss{n}{k}\le(q+\beta)q^{k(n-k)-1}.
\end{align*}
where $\beta=5$ for $q=2$, and $\beta=3$ for $q=3$ and $\beta=2$ for $q\ge 4$. Also, if  $0<k<n$, then $(q+1)q^{k(n-k)-1}\le\gauss{n}{k}$.
\end{lemma}
\begin{proof}
We first prove the lower bound for the Gaussian coefficient for $0<k<n$ by induction on $k$. For $k=1$ we have $$\gauss{n}{k}=\sum_{i=0}^{n-1}q^i\ge(q+1)q^{n-2}$$ as required. For $k\ge 2$, we have $$\gauss{n}{k}=\frac{q^n-1}{q^k-1}\gauss{n-1}{k-1}\ge q^{n-k}\gauss{n-1}{k-1}$$ and the induction hypothesis shows that this is at least $(q+1)q^{k(n-k)-1}$.

For $q\ge 3$, the upper bound was proved in \cite[Lemma 34]{Ihringer&Metsch}. Now consider $q=2$. We have to prove the upper bound, which we do using the technique of \cite[Lemma 34]{Ihringer&Metsch}. From the definition of the Gaussian coefficient, we have
\begin{align}
    \gauss{n}{k} &= \prod_{i=1}^k \frac{q^{n-k+i}-1}{q^i-1} \leq \prod_{i=1}^k \frac{q^{n-k+i}}{q^i-1} = q^{k(n-k)} \prod_{i=1}^k \frac{q^i}{q^i-1}.\label{lem_upper_gauss_ineq1}
\end{align}
It suffices therefore to show that $\prod_{i=1}^k \frac{q^i}{q^i-1}\le 1+\frac\beta q=\frac72$. For $1\le k\le 5$, one verifies this by hand. For $k\ge 6$, one proves easily by induction on $k$ the stronger statement $\prod_{i=1}^k \frac{q^i}{q^i-1}\le 1+5\cdot\frac{q^{k-1}-2}{q^k-2}$.
\end{proof}

\begin{lemma}\label{numberofpairs}
Let $K_1$ and $K_2$ be two subspaces of dimension $k$ of an $F_q$-vector space and let $s$ be the dimension of their intersection.
Then for every integer $i$ with $0\le i\le s$ the number of pairs $(T_1,T_2)$ consisting of a $t$-subspace $T_1$ of $K_1$ and a $t$-subspace $T_2$ of $K_2$ such that $\dim(T_1\cap T_2)=i$ is at most
\begin{align*}
\gauss{s}{i}\gauss{k-i}{t-i}^2.
\end{align*}
\end{lemma}
\begin{proof}
For any such pair, the intersection of $T_1$ and $T_2$ is a subspace of the $s$-subspace $K_1\cap K_2$. Now $K_1\cap K_2$ has $\gauss{s}{i}$ $i$-subspaces and each such $i$-subspace lies in $\gauss{k-i}{t-i}$ $t$-subspaces of $K_1$ and in as many of $K_2$. The statement follows.
\end{proof}

\begin{lemma}\label{parabola}
Let $f$ be a real quadratic polynomial with leading coefficient $-1$, and let $f$ obtain its maximum for $x=x_0$. Let $q$ and $a$ be integers and suppose that $q\ge 2$.
\begin{enumerate}[(a)]
\item If $x_0\le a$, then $\sum_{i=a}^\infty q^{f(i)}<q^{f(a)}(1+\frac1q+\frac1{q^3})$.
\item If $x_0\ge a$, then $\sum_{i=-\infty}^aq^{f(i)}<q^{f(a)}(1+\frac1q+\frac1{q^3})$.
\item If $2x_0$ is an integer, then $\sum_{i\in\mathbb{Z}}q^{f(i)}< q^{f(x_0)}(1+\frac2q+\frac2{q^3})$.
\end{enumerate}
\end{lemma}
\begin{proof}
We use several times that $\sum_{i=0}^\infty q^{-i}=q/(q-1)\le q$.

(a) Since $x_0\le a$, then $f(a+i)\le f(a)-i^2$ for all $i\ge 0$ and hence
\begin{align*}
\sum_{i=a}^\infty q^{f(i)}\le q^{f(a)}\sum_{i=0}^\infty q^{-i^2}
<q^{f(a)}\left(1+\frac1q+q^{-4}\sum_{i=0}^\infty q^{-i}\right)
\le q^{f(a)}\left(1+\frac1q+\frac1{q^3}\right).
\end{align*}

(b) This follows from (a) and the symmetry of $f$.

(c) If $x_0$ is an integer, this follows from (a) and (b) applied with $a=x_0$. Now consider the case when $x_0$ is not an integer but $2x_0$ is. If $z\in\mathbb{Z}$, then $i:=z-x_0-\frac12\in\mathbb{Z}$ and $f(z)=f(x_0)-(\frac12+i)^2$. Hence
\begin{align*}
\sum_{z\in\mathbb{Z}} q^{f(z)}&=q^{f(x_0)}\sum_{i\in\mathbb{Z}}q^{-(1/2+i)^2}
=2q^{f(x_0)-\frac14}\sum_{i=0}^\infty q^{-i(i+1)}
\\
& \le2q^{f(x_0)-\frac14}(1+q^{-2}+q^{-6}\sum_{i=0}^\infty q^{-i})
\\
&
\le 2q^{f(x_0)-\frac14}(1+q^{-2}+q^{-5})
\\
&<q^{f(x_0)}(1+\frac2q+\frac2{q^3})
\end{align*}
where the last step is obvious for $q\ge 16$ and easily checked for $2\le q\le 15$.
\end{proof}

The next theorem is a reformulation of Theorem \ref{main1}.

\begin{theorem}\label{mainwork}
Let $\Gamma=K_q(n,k,t)$ with $n\ge 2k$ and $k>t\ge 1$. Assume that $\tw(\Gamma)\not=\gauss{n}{k}-\alpha(\Gamma)-1$. Then the following hold where $\epsilon$ is the function defined in the introduction.
\begin{enumerate}[(a)]
\item If $t\le \epsilon(q)$, then $n\le 3k-2t+\epsilon(q)$.
\item If $t>\epsilon(q)$, then $n\le 3k-t+1-2\sqrt{t-\epsilon(q)}$.
\end{enumerate}
\end{theorem}
\begin{proof}
Lemma \ref{Robertson&Seymour} shows that there exists a subset $P$ of $V(\Gamma)$ such that $|P|\le \tw(\Gamma)+1$ and such that every component of $\Gamma\setminus P$ has at most $\frac12|V(\Gamma)\setminus P|$ vertices. Consider $Y:=V(\Gamma)\setminus P$.

Since $\tw(\Gamma)\not=\gauss{n}{k}-\alpha(\Gamma)-1$, Corollary \ref{upperbound} shows that $\tw(\Gamma)\le \gauss{n}{k}-\alpha(\Gamma)-2$. Since $|P|\le \tw(\Gamma)+1$, it follows that $|Y|\ge \alpha(\Gamma)+1$. Hence $Y$ is not an independent set of $\Gamma$ and thus $\Gamma$ has an edge whose vertices $S_1$ and $S_2$ are in $Y$. Let $X$ be the vertex set of the component of $\Gamma\setminus P$ that contains $S_1$ and $S_2$. Then $|X|\le \frac12|Y|$. Hence $|Y\setminus X|\ge \frac12|Y|>\frac12\alpha(\Gamma)$.

As vertices of the graph $\Gamma$, $S_1$ and $S_2$ are connected, so as $k$-subspaces we have $s:=\dim(S_1\cap S_2)\le t-1$.
As vertices of the graph, the elements of $Y\setminus X$ are neither adjacent to $S_1$ nor to $S_2$, so as $k$-subspaces, every element of $Y\setminus X$ meets each of the $k$-subspaces $S_1$ and $S_2$ in a subspace of dimension at least $t$. We define
\begin{align*}
M&:=\{(T_1,T_2)\mid \text{$T_i$ is a subspace of $S_i$ with $\dim(T_i)=t$, $i=1,2$}\}
\\
C&:=\{((T_1,T_2),K)\in M\times(Y\setminus X)\mid T_1,T_2\subseteq K\}.
\end{align*}
As every element of $Y\setminus X$ meets $S_1$ and $S_2$ in subspaces of dimension at least $t$, then every element of $Y\setminus X$ occurs in a pair of $C$ and hence $|C|\ge |Y\setminus X|>\frac12\alpha(\Gamma)$.

Consider an element $(T_1,T_2)$ of $M$ that occurs in at least one pair of $C$. Then $\dim(T_1+T_2)\le k$ and hence $\dim(T_1\cap T_2)\ge 2t-k$. As $T_1\cap T_2\subseteq S_1\cap S_2$ then $\dim(T_1\cap T_2)\le s$. Hence, if $i=\dim(T_1\cap T_2)$, then $2t-k\le i\le s$ and the number of $k$-subspaces of $F_q^n$ that contain $T_1$ and $T_2$ is $\gauss{n-2t+i}{k-2t+i}$; consequently $(T_1,T_2)$ lies in at most this many pairs of $C$. Hence with
\begin{align*}
i_{\max}&:=\max\{0,2t-k\}
\end{align*}
we deduce from \ref{Frankl&Wilson} and \ref{numberofpairs} that
\begin{align}\label{eqn_asdrfe}
\frac12\gauss{n-t}{k-t}\stackrel{\ref{Frankl&Wilson}}{=}\frac12\alpha(\Gamma)<|C|\stackrel{\ref{numberofpairs}}{\le} \sum_{i=i_{\max}}^s\gauss{s}{i}\gauss{k-i}{t-i}^2\gauss{n-2t+i}{k-2t+i}.
\end{align}
We define
\begin{align*}
\beta:=\begin{cases}
5 & \text{if } q=2,
\\
3 & \text{if } q=3,
\\
2 & \text{if } q\ge 4,
\end{cases}
\end{align*}
and apply Lemma \ref{guassestimate} to both sides of inequality \eqref{eqn_asdrfe} to find
\begin{align*}
\frac12(q+1)q^{(n-k)(k-t)-1}\le \sum_{i=i_{\max}}^s(q+\beta)^4q^{(s-i)i-1}(q^{(k-t)(t-i)-1})^2q^{(n-k)(k-2t+i)-1}.
\end{align*}
Using $s\le t-1$ it follows that
\begin{align*}
\frac{(q+1)q^3}{2(q+\beta)^4}\le &\sum_{i=i_{\max}}^sq^{(s-i)i}q^{2(k-t)(t-i)}q^{(n-k)(-t+i)}
\\
\le &\sum_{i=i_{\max}}^{t-1}q^{(t-1-i)i}q^{2(k-t)(t-i)}q^{(n-k)(-t+i)}
\\
= &\sum_{i=i_{\max}}^{t-1}q^{(t-i)(i+3k-2t-n)-i}.
\end{align*}
Using the definition of $\beta$ and $\epsilon(q)$ it is straightforward to check that
\begin{align}
q^{-\epsilon(q)-\frac34}(1+\frac2q+\frac2{q^3})<\frac{(q+1)q^3}{2(q+\beta)^4}.
\end{align}
Defining the quadratic function $f(i)=(t-i)(i+3k-2t-n)-i$ we thus find
\begin{align}\label{eqncrucial}
q^{-\epsilon(q)-\frac34}(1+\frac2q+\frac2{q^3})<\sum_{i=i_{\max}}^{t-1}q^{f(i)}.
\end{align}
The parabola $f$ reaches its maximum for $i=i_0$ where
\begin{align}\label{fi0}
i_0:=\frac12(n-3k+3t-1)\ \text{and} \ \ f(i_0)=\frac14(3k+1-t-n)^2-t.
\end{align}
Lemma \ref{parabola} (c) implies that the right hand side of \eqref{eqncrucial} is at most $(1+\frac2q+\frac2{q^3})q^{f(i_0)}$. Therefore \eqref{eqncrucial} implies that $f(i_0)>-\frac34-\epsilon(q) $. Since $f(i_0)$ or $f(i_0)-\frac 14$ is an integer, it follows that $f(i_0)\ge -\epsilon(q)$. If $i_0\le i_{\max}$ or $i_0\ge t-1$, we can improve this by applying (a) or (b) of Lemma \ref{parabola}. Since $f(i_{\max})$ and $f(t-1)$ are integers, we find
\begin{align}\label{cases}
-\epsilon(q)\le
\begin{cases}
f(i_0) & \text{in any case,}
\\
f(t-1) & \text{if $t-1\le i_0$,}
\\
f(i_{\max}) & \text{$i_0\le i_{\max}$.}
\\
\end{cases}
\end{align}
Case 1. Here we consider the situation when $n\ge 3k-t-1$.

Then \eqref{fi0} gives $i_0\ge t-1$ so \eqref{cases} implies that $-\epsilon(q)\le f(t-1)$. Since $f(t-1)=3k-2t-n$, it follows that $n\le 3k-2t+\epsilon(q)$. Hence $3k-t-1\le n\le 3k-2t+\epsilon(q)$, which implies that $t\le \epsilon(q)+1$. If $t\le \epsilon(q)$, then we are in situation (a) of the statement, and if $t=\epsilon(q)+1$, we are situation (b) of the statement. Hence, in this case the theorem is proved.

Case 2. Here we consider the case that $n\le 3k-t-2$.

If $t\le \epsilon(q)$, then we have $n\le 3k-t-2\le 3k-2t+\epsilon(q)$ and we are in situation (a) of the statement. Suppose now that $t>\epsilon(q)$. From \eqref{cases} we find $-\epsilon(q)\le f(i_0)$, that is
    \begin{align*}
    t-\epsilon(q)\le \frac14(3k+1-t-n)^2.
    \end{align*}
Since $n\le 3k-t-2$ and $t>\epsilon(q)$, it follows that $n\le 3k-t+1-2\sqrt{t-\epsilon(q)}$. Now we are situation (b) of the statement. Hence, also in this case the theorem is proved.
\end{proof}

Theorem \ref{main1} follows immediately from Theorem \ref{mainwork}. Corollary \ref{maincor1} follows from Theorem \ref{main1}.

{\bf Proof of Corollary \ref{maincor2}}. Consider a Kneser graph $K_q(n,k,t)$ as in \ref{maincor2}, that is with $q\ge 9$ and $t>k+3-2\sqrt{k+2}$. From $t>k+3-2\sqrt{k+2}$ it follows that $k-t+1<2\sqrt{t}$. Hence in part (b) of Theorem \ref{main1} we have $3k-t+1-2\sqrt{t}\le 2k$. Therefore \ref{maincor2} follows from Theorem \ref{main1}.

\section{The Kneser graph $K_q(4,2,1)$}

The graph $K_q(n,k,t)$ with $t=k-1$ is the complement of a Grassmann graph. This graph attracted special attention in \cite{CaoLiuLuLv} where it was shown for all $n\ge k\ge 2$ with the exception $(n,k)=(4,2)$ that the graph $K_q(n,k,k-1)$ has treewidth $|V(\Gamma)|-\alpha(\Gamma)-1$. The case $K_q(4,2,1)$ remained unsolved in \cite{CaoLiuLuLv} and also our counting argument from the last section is not strong enough to cover this case. We will choose a special model for this graph and use geometric properties to determine its treewidth. The case $q=2$ will lead to extra difficulties, whereas the case $q\ge 3$ only requires some basic knowledge of tree decompositions.

Using projective geometry, the graph $K_q(4,2,1)$ can be understood as the graph whose vertices are the lines of the projective space $\PG(3,q)$ with two vertices adjacent if and only if the lines are skew. Using the Klein-correspondence from $\PG(3,q)$ to the hyperbolic quadric $Q^+(5,q)$, we can define the same graph as follows.
Its vertices are the points of $Q^+(5,q)$ and two vertices are non-adjacent if and only if the lie on a line of $Q^+(5,q)$. We consider $Q^+(5,q)$ naturally embedded in $\PG(5,q)$ and denote the related polarity of $\PG(5,q)$ by $\perp$. Then two points $v$ and $v'$ of $Q^+(5,q)$ are adjacent as vertices of the graph if and only if the points are not perpendicular, that is $v'$ is not a point of the tangent hyperplane $v^\perp$ at $v$. See \cite{Hirschfeld&Thas} for properties of the hyperbolic quadric and the fact that this graph is indeed isomorphic to $K_q(4,2,1)$. Using this model we will show that $\tw(\Gamma)=|V(\Gamma)|-\alpha(\Gamma)-1$. We need two lemmata of preparation, the first one is an easy observation of the hyperbolic quadric $Q^+(3,q)$, which occurs in $Q^+(5,q)$ as an intersection of $Q^+(5,q)$ with $\ell^\perp$, where $\ell$ is any secant line to the quadric.

\begin{lemma}\label{hyperbolicquadric}
The largest sets of points of the hyperbolic quadric $Q^+(3,q)$ without three pairwise non-collinear points are the unions of two skew lines with $2q+2$ points.
\end{lemma}
\begin{proof}
Notice that the hyperbolic quadric $Q^+(3,q)$ is a   $(q+1)\times (q+1)$-grid, so it can be described as follows: Its points are $(i,j)$ for $0\le i,j\le q$ and its lines are $\{(i,j)\mid 0\le i\le q\}$ and $\{(j,i)\mid 0\le j\le q\}$ for $0\le i\le q$. The assertion is easily proved from this.
\end{proof}

The next lemma collects some properties of tree decompositions.

\begin{lemma}\label{treeproperties}
Let $(T,B)$ with $B=(B_t\mid t\in V(T))$ be a tree decomposition of a finite graph $\Gamma$. Let $t$ be a vertex of $T$. Then we have the following.
\begin{enumerate}[(a)]
\item If $t\in V(T)$ and $v\in V(\Gamma)\setminus B_t$, then $T\setminus\{t\}$ has exactly one component such that $v\in B_s$ for some vertex $s$ of this component. We denote the vertex set of this component by $T_t(v)$ and we denote by $\Gamma_t(v)$ the union of the sets $B_s$ with $s\in T_t(v)$.
\item For adjacent vertices $v,w$ of $\Gamma$ with $v\notin B_t$, we have $w\in \Gamma_t(v)$.
\item For adjacent vertices $v,w$ of $\Gamma$ with $v,w\notin B_t$ we have $\Gamma_t(v)=\Gamma_t(w)$.
\item If $v$ is a vertex of a component $C$ of $\Gamma\setminus B_t$, then $\Gamma_t(v)$ contains every vertex of $C$ and every vertex of $\Gamma$ that has a neighbor in $C$.
\end{enumerate}
\end{lemma}
\begin{proof}
Part (a) follows from the second property of the definition of a tree decomposition. Part (b) then follows from the fact that the edge $\{v,w\}$ of $\Gamma$ is contained in $B_s$ for some vertex $s$ of $T$, again by the definition of a tree decomposition. Part (c) follows from (a) and (b). For part (d), we first apply (c) to see that $C\subseteq \Gamma_t(v)$ and then (b) to see that every vertex with a neighbor in $C$ is contained in $\Gamma_t(v)$.
\end{proof}

\begin{theorem}
The treewidth of $K_q(4,2,1)$ is $|V(\Gamma)|-(q^2+q+2)$ for all prime powers $q$.
\end{theorem}
\begin{proof}
We represent $\Gamma:=K_q(n,2,1)$ by $Q^+(5,q)$ as explained above. We also consider the ambient projective space $\PG(5,q)$ of $Q^+(5,q)$ and the related polarity $\perp$. For each set $X$ of points of $Q^+(5,q)$, we denote by $\erz{X}$ the subspace of $\PG(5,q)$ that is spanned by the points of $X$. The set $\erz{X}^\perp\cap Q^+(5,q)$ consists of the points of $Q^+(5,q)$ that are perpendicular to all points of $X$. In $\Gamma$ the set $\erz{X}^\perp\cap Q^+(5,q)$ consists of the vertices that are not adjacent to any vertex in $X$ (notice that vertices of $X$ may lie in this set).

From Result \ref{Frankl&Wilson} we have $\alpha(\Gamma)=q^2+q+1$ and Corollary \ref{upperbound} shows that $\tw(\Gamma)\le |V(\Gamma)|-\alpha(\Gamma)-1$. Assume that
\begin{align*}
\tw(\Gamma)\le |V(\Gamma)-\alpha(\Gamma)-2\le |V(\Gamma)|-q^2-q-3.
\end{align*}
We shall derive a contradiction.

Lemma \ref{Robertson&Seymour} shows that there exists a tree decomposition $(T,B)$, $B=(B_t\mid t\in V(T))$, of $\Gamma$, a vertex $t_0\in V(T)$ and a subset $P$ of $B_{t_0}$ with the following properties
\begin{enumerate}[(T1)]
\item The tree decomposition $(T,B)$ has width $\tw(\Gamma)$.
\item For adjacent vertices $s$ and $t$ of $T$ we have $B_s\not\subseteq B_t$.
\item With $Y:=V(\Gamma)\setminus P$, every component of the graph $\Gamma_Y$ induced by $\Gamma$ on $Y$ has at most $\frac12|Y|$ vertices.
\end{enumerate}
As $|P|\le |B_{t_0}|\le\tw(\Gamma)+1$, then
\begin{align}\label{boundY}
|Y|=|V(\Gamma)|-|P|\ge |V(\Gamma)|-|B_{t_0}|\ge |V(\Gamma)|-\tw(\Gamma)-1\ge q^2+q+2.
 \end{align}
Hence $|Y|>\alpha(\Gamma)$ and thus $Y$ contains adjacent vertices $v_1$ and $v_2$. Let $X$ be the vertex set of the component of $\Gamma_Y$ that contains $v_1$ and $v_2$. Then no vertex of $Y\setminus X$ is adjacent to any vertex of $X$.
Since $X$ is a component of $\Gamma_Y$, we have $|X|\le\frac12|Y|$ and hence $|Y\setminus X|\ge \frac12|Y|\ge \frac12(q^2+q+2)>q+1$.

On $Q^+(5,q)$ the set $Y$ is a set of points, $X$ is a subset of $Y$, $v_1$ and $v_2$ are non-perpendicular points of $X$, and every point of $Y\setminus X$ is distinct and perpendicular to every point of $X$. Therefore $Y\setminus X$ is a subset of $\erz{X}^\perp\cap Q^+(5,q)$. Hence the subspaces $\erz{X}$ and $\erz{Y\setminus X}$ are perpendicular.

Case 1. There exists three pairwise non-perpendicular points in $Y$.

We may assume that these are $v_1$ and $v_2$ and a third point $v_3$. Then $v_3\in X$ and $v_1,v_2,v_3$ are pairwise non-collinear points of $Q^+(5,q)$. Therefore $\pi:=\erz{v_1,v_2,v_3}$ is a conic plane, that is a plane of $\PG(5,q)$ that meets $Q^+(5,q)$ in the $q+1$ points of a conic. We have  $Y\setminus X\subseteq \erz{X}^\perp\subseteq \pi^\perp$. As $\pi$ is a conic plane, then $\pi^\perp$ is also a conic plane and hence has $q+1$ points on $Q^+(5,q)$. Therefore $|Y\setminus X|\le |\pi^\perp\cap Q^+(5,q)|=q+1$. But we have seen above that $|Y\setminus X|>q+1$. This is a contradiction.

Case 2. $Y$ does not contain three pairwise non-perpendicular points.

The line $\ell$ of $\PG(5,q)$ on $v_1$ and $v_2$ is a secant line of $Q^+(5,q)$ and hence $\ell^\perp$ is a 3-space of $\PG(5,q)$ that meets $Q^+(5,q)$ in a hyperbolic quadric $Q^+(3,q)$. We have $Y\setminus X\subseteq \erz{X}^\perp\subseteq\ell^\perp$, so $Y\setminus X$ is contained in this hyperbolic quadric and does not contain three pairwise non-collinear points. As $|Y\setminus X|>q+1$, then $\erz{Y\setminus X}$ is either the solid $\ell^\perp$ or a plane of this solid.

Case 2.1. $\erz{Y\setminus X}$ is the solid $\ell^\perp$.

Then $Y\setminus X$ is a subset of $\ell^\perp\cap Q^+(5,q)$, which is the hyperbolic quadric $Q^+(3,q)$ already mentioned above. Lemma \ref{hyperbolicquadric} shows therefore that $Y\setminus X$ contains at most $2(q+1)$ points. As $\erz{Y\setminus X}=\ell^\perp$, then $\ell=\erz{Y\setminus X}^\perp$ and since $X$ and $Y\setminus X$ are perpendicular, if follows that $X\subseteq \ell$, that is $X=\{v_1,v_2\}$. It follows that $|Y|=|X|+|Y\setminus X|\le 2+2(q+1)$. Therefore \eqref{boundY} implies that $q=2$ and  $|Y|=8$ and $|X|=2$. Hence $Y\setminus X$ consists of six points of $Q^+(3,q)$ and Lemma \ref{hyperbolicquadric} shows that these points are the points of two skew lines of $Q^+(3,q)$. The graph induced by $\Gamma$ on $Y\setminus X$ is therefore a 6-cycle. Since $Y$ has eight points, this contradicts property (T3) of the tree decomposition $(T,B)$.

Case 2.2 $\pi:=\erz{Y\setminus X}$ is a plane of $\ell^\perp$.

As $|Y|>q+1$, then $\pi$ meets $Q^+(5,q)$ in the union of two lines. Let $z$ be the point of intersection of these two lines. The subspace $\pi^\perp$ is also a plane that meets $Q^+(5,q)$ in the union of two lines and $z$ is also their intersection point. Since two intersecting lines contain $2q+1$ points, then $X$ and $Y\setminus X$ have at most $2q+1$ points, but since $z$ belongs to at most one of these sets, we have $|Y|=|X|+|Y\setminus X|\le 4q+1$. Since $|Y|\ge q^2+q+2$, it follows that $q=2$ and $|Y|\in\{8,9\}$ and \eqref{boundY} can be written as follows
\begin{align}\label{boundY2}
|V(\Gamma)|-9\le |V(\Gamma)|-|Y|=|P|\le |B_{t_0}|\le \tw(\Gamma)+1\le |V(\Gamma)|-8.
\end{align}
Clearly in \eqref{boundY2} all inequalities are sharp except one that is missing sharpness by one.
Recall that $V(\Gamma)\setminus Y=P\subseteq B_{t_0}$, so either $P=B_{t_0}$ and $B_{t_0}\cap Y=\emptyset$ or otherwise $|B_{t_0}|=|P|+1$ and $|Y\cap B_{t_0}|=1$.

Denote by $\bar Y$ the set consisting of the $4q+1=9$ points of $Q^+(5,q)$ in $\pi\cup \pi^\perp$, which are $z$ and eight points of the two $4$-cycles. Then $Y\subseteq \bar Y$. From \eqref{boundY2} we have $|V(\Gamma)\setminus B_0|\in\{8,9\}$. Since $V(\Gamma)\setminus B_{t_0}\subseteq V(\Gamma)\setminus P=Y\subseteq \bar Y$, it follows that $|B_{t_0}\cap \bar Y|\le 1$. The graph induced by $\Gamma$ on $\bar Y$ is the disjoint union of the singleton $z$ and two 4-cycles, each 4-cycle consisting of the four points of $Q^+(5,q)$ other than $z$ of one of the planes $\pi$ and $\pi^\perp$. We refer to these two $4$-cycles as \emph{the $4$-cycles} for the rest of the proof. We denote a $4$-cycle by $uvwx$ and mean hereby that $u\sim v\sim w\sim x\sim u$ but there are no other adjacencies between the vertices $u,v,w$ and $x$; in $Q^+(5,q)$, the points $z,v,x$ as well as the points $z,w,u$ are the points of the two lines of $\pi$ or $\pi^\perp$ on $z$. Notice that any three vertices of the 4-cycle in $\pi$ span $\pi$, so that every point that is perpendicular to these three vertices lies in $\pi^\perp$. The same holds for the $4$-cycle n $\pi^\perp$. Hence we have:

(E) Given three vertices of one of the $4$-cycles, the only vertices that have no neighbor among these three vertices are $z$ and the vertices of the second $4$-cycle.

Case 2.2.1: $z\notin B_{t_0}$.

We have $|B_{t_0}\cap \bar Y|\le 1$. Hence, if $B_{t_0}$ contains a vertex of $\bar Y$, then $z\notin B_{t_0}$ implies that this vertex belongs to one of the two 4-cycles. If $B_{t_0}$ contains no vertex of $\bar Y$, then we let $uvwx$ be any of the two 4-cycles, and otherwise we let $uvwx$ be the $4$-cycle with a vertex in $B_{t_0}$ and number the vertices of the $4$-cycle in such a way that $x\in B_{t_0}$. Then $u,v,w\notin B_{t_0}$.

Consider the components $T_{t_0}(u)$, $T_{t_0}(v)$ and $T_{t_0}(w)$
defined in Lemma \ref{treeproperties}. Since $u\sim v\sim w$, part (c) of Lemma \ref{treeproperties} shows that these three components are the same, and hence the corresponding sets $\Gamma_{t_0}(u)$, $\Gamma_{t_0}(v)$ and $\Gamma_{t_0}(w)$ defined in Lemma \ref{treeproperties} are the same. Since the only vertices that have no neighbor in $\{u,v,w\}$ are $z$ and the four vertices of the second $4$-cycle and since these five vertices do not belong to $B_{t_0}$, part (b) of Lemma \ref{treeproperties} shows that $B_{t_0}\subseteq \Gamma_{t_0}(u)$. Let $t$ be the unique vertex of the component $B_{t_0}(u)$ that is adjacent to $t_0$ in $T$. Since $B_0\subseteq \Gamma_{t_0}(u)$ and since $(T,B)$ is a tree decomposition of $\Gamma$, property (TD2) of the definition of a tree decomposition implies that $B_{t_0}\subseteq B_t$. This contradicts property (T2) of the tree decomposition $(T,B)$.

Case 2.2.2. $z\in B_{t_0}$.

We have seen that $|Y|\in\{8,9\}$. If $|Y|=8$, then (10) shows that $P=B_{t_0}$ and $|B_{t_0}|=\tw(\Gamma)+1=|V(\Gamma)|-8$. Since $P\cap Y=\emptyset$, we have $z\notin Y$ in this case. If $|Y|=9$, then $z\in \bar Y=Y$ and hence $z\notin P$, so (10) shows that $B_{t_0}=P\cup\{z\}$ and $|B_{t_0}|=\tw(\Gamma)+1=|V(\Gamma)|-8$.

In any case we have $|B_{t_0}|=\tw(\Gamma)+1=|V(\Gamma)|-8$ and the eight vertices that are not in $B_{t_0}$ are the eight vertices of the two $4$-cycles and these eight vertices lie in $Y$. Put $E:=B_{t_0}\setminus \{z\}$. Then $|E|=|V(\Gamma)|-9=\tw(\Gamma)$ and $B_{t_0}=E\cup\{z\}$.

Let $uvwx$ be one of the two $4$-cycles. As before, Lemma \ref{treeproperties} shows that the sets $T_{t_0}(a)$ with $a$ in this $4$-cycle are all the same, and hence that the corresponding sets $\Gamma_{t_0}(a)$ are all the same. Lemma \ref{treeproperties} shows that every vertex which has a neighbor in the $4$-cycle $uvwx$ lies in $\Gamma_{t_0}(a)$. Since the only vertices with no neighbor in this $4$-cycle are $z$ and the four vertices of the second $4$-cycle, it follows that $E\subseteq \Gamma_{t_0}(a)$. Let $t$ be the vertex of the component of $T_{t_0}(u)$ that is adjacent to $t_0$.

We next show that $E\subseteq B_t$. To see this, consider $e\in E$. Then $e\in\Gamma_{t_0}(w)$ and hence $e\in B_i$ for some $i\in T_i(w)$. Since also $e\in B_{t_0}$ and since $t$ lies on the unique path of the tree $T$ from $t$ to $i$, it follows that $e\in B_t$. As this holds for any $e\in E$, we find $E\subseteq B_t$.

As $B_{t_0}=E\cup\{z\}$, then (T2) implies that $B_t\not=E$ and $z\notin B_t$.  Since the tree representation under consideration has treewidth $\tw(\Gamma)+1$, then $|B_t|\le \tw(\Gamma)+1=|E|+1$. It follows that there exists a unique vertex $r$ with $B_t=P\cup\{r\}$. Also $r\not=z$ and $r\notin E$.

Since $r\notin B_{t_0}$, then $r$ is a vertex of one of the two 4-cycles. If it is not a vertex of the $4$-cycle $uvwx$ but of the other 4-cycle, then the other 4-cycle is also contained in $\Gamma_{t_0}(u)$. We may thus assume that $r=u$.

Then $v,w,x\notin B_t$. Consider $\Gamma_t(w)$ and let $s$ be the unique vertex of the component $T_t(w)$ that is adjacent to $t$. Since $vwx$ is a path of length two, then Lemma \ref{treeproperties} shows that $v,x\in\Gamma_t(w)$ and $\Gamma_t(w)=\Gamma_t(v)=\Gamma_t(x)$. By property (E), every vertex of $E$ is adjacent to at least one vertex of the path $vwx$. Therefore Lemma \ref{treeproperties} shows that $E\subset \Gamma_t(w)$. As $E\subseteq B_t$, the argument used above to show that $E\subseteq B_t$ shows now that $E\subseteq B_s$. As $B_t=E\cup\{u\}$, then (T2) shows that $u\notin B_s$.

As $|E|=|V(\Gamma)|-9$ and $|B_s|\le\tw(\Gamma)+1=|V(\Gamma)|-8$, it follows that $B_s$ contains at most one of the vertices $v$ and $x$, we may assume w.l.o.g. that $v\notin B_s$. Recall that $v$ lies in $\Gamma_t(w)$, which is the union of the sets $B_i$ with vertices $i$ of $T_t(w)$. As $v\notin B_s$, then the fact that $(B,T)$ is a tree decomposition implies that every vertex $i$ of $T$ with $v\in B_i$ must be a vertex of $T_t(w)$. Since $\{u,v\}$ is an edge, there exists a vertex $i$ of $T$ with $u,v\in B_i$. Then $i$ is a vertex of $T_t(w)$. We have $u\in T_t$ and $u\in T_i$. But the path of $T$ from $t$ to $i$ contains the vertex $s$ and $u\notin T_s$. Hence the graph induced by $T$ on the vertices $j$ with $u\in B_j$ is disconnected. This contradicts the fact that $(B,T)$ is a tree decomposition of $\Gamma$.
\end{proof}

\bibliographystyle{plain}

\end{document}